\documentclass[12pt]{amsart}
\pdfoutput=1
\usepackage{microtype}
\usepackage[active]{srcltx}
\usepackage{hyperref}
\usepackage{calc,amssymb,amsthm,amsmath,amscd, eucal,ulem, mathtools}
\usepackage{phaistos}
\usepackage{alltt}
\RequirePackage[dvipsnames,usenames]{xcolor}

\input{mabliautoref.sty}
\input{xy}
\xyoption{all}
\input{kmacros3.sty}
\usepackage{tikz,calc}
\usepackage{tikz-cd}
\usetikzlibrary{external}

\usepackage[marvosym]{tikzsymbols}
\usepackage{amsfonts, mathrsfs}
\usepackage[cal=boondox, calscaled=1.05]{mathalfa}
\usepackage{calligra}
\usepackage{stmaryrd} %power series bracket
\usepackage{dcpic, pictexwd}
\usepackage[left=1in,top=1in,right=1in,bottom=1in]{geometry}
\usepackage{bm}
\usepackage{verbatim}
\usepackage{upgreek}

\numberwithin{equation}{theorem}

\renewcommand{\:}{\colon}

\newcommand{\p}{\mathfrak{p}}
\newcommand{\kay}{\mathcal{k}}
\newcommand{\el}{\mathcal{l}}

\DeclareMathOperator{\RHom}{RHom}

\usepackage{setspace}
\usepackage{enumerate}
\usepackage{graphicx}
\usepackage[all,cmtip]{xy}

\usepackage{verbatim}

\renewcommand{\sF}{\mathcal{F}}

\renewcommand{\sH}{\mathcal{H}}

\renewcommand{\sO}{\mathcal{O}}

\usepackage{marvosym}

\begin{document}
\title{Adjoint test modules along Cohen--Macaulay morphisms} 
\author[J.~Carvajal-Rojas]{Javier Carvajal-Rojas}
\address{Centro de Investigaci\'on en Matem\'aticas, A.C., Callej\'on Jalisco s/n, 36023 Col. Valenciana, Guanajuato, Gto, M\'exico}
\email{\href{mailto:javier.carvajal@cimat.mx}{javier.carvajal@cimat.mx}}
\author[A.~St\"abler]{Axel St\"abler}
\address{Universit\"at Leipzig\\ Mathematisches Institut\\ Augustusplatz 10\\
04109 Leipzig\\Germany} 
\email{\href{mailto:staebler@math.uni-leipzig.de}{staebler@math.uni-leipzig.de}}


\thanks{Carvajal-Rojas was partially supported by the grants SECIHTI \#CBF2023-2024-224, \#CF-2023-G-33, and \#CBF-2025-I-673.}


\begin{abstract}
We provide a transformation rule for adjoint test modules along Cohen--Macaulay maps between Cohen--Macaulay varieties that have $F$-rational geometric fibers. This is, in part, an effective version of Enescu's theorem on the ascent of $F$-rationality under local maps with $F$-rational geometric fibers.
\end{abstract}
\maketitle

\section{Introduction}

Rational singularities are the gold standard for mild, nice singularities in algebraic geometry. Under reduction mod $p$, rationality corresponds to \emph{$F$-rationality} in the theory of $F$-singularities. A key foundational result in this theory is that of Florian~Enescu \cite[Theorem 2.24]{EnescuOnTheBehaviorOfFrationalRingsUnderFlatBaseChange}. Enescu showed that, under mild technical hypotheses, $F$-rationality ascends along flat local homomorphisms $\theta \: (R,\fram,\kay) \to (S,\fran,\el)$ with geometrically $F$-rational fibers; i.e., $\kay' \otimes_{R} S$ is $F$-rational for all field extensions $\kay'/\kay$. However, there are a few ways in which one can measure $F$-rationality, and so it is natural to wonder about transformation rules for these invariants that serve as effective generalizations of Enescu's theorem. Here, we focus on adjoint test modules.

Recall that $F$-rational singularities are normal and Cohen–Macaulay singularities. To see what they are among them, let us fix an $F$-finite base field $\kay$ (e.g., perfect), and consider a normal Cohen–Macaulay $\kay$-variety $X$ with canonical sheaf $\omega_X$. It comes equipped with the so-called \emph{Cartier operator} $\kappa\: F_* \omega_X \to \omega_X$, and so $(\omega_X,\kappa)$ is a Cartier module on $X$. The \emph{adjoint test module} of $X$ is $\uptau(\omega_X)\subset \omega_X$ and is defined as the smallest non-trivial Cartier $\sO_X$-submodule of $(\omega_X,\kappa)$, i.e., the smallest non-trivial submodule $M\subset \omega_X$ such that $\kappa(F_* M) \subset M$. Then, $X$ has $F$-rational singularities if and only if $\uptau(\omega_X) = \omega_X$. Thus, the annihilator ideal of $\omega_X/\uptau(\omega_X)$ cuts out the non-$F$-rational locus of $X$, thereby measuring the $F$-rationality of $X$. The above can be relaxed to the case where $X$ is a Cohen–Macaulay reduced  $F$-finite scheme, at the expense of a more technical definition; see \autoref{def.canonicalsheafandtestmodule}. The goal of this note is to prove the following transformation rule.

\begin{mainthm*}
\label{CMfratfibersthm}
Let $X$ be a reduced $F$-finite Cohen--Macaulay scheme. Let $f\colon Y \to X$ be a finite type Cohen--Macaulay map of pure relative dimension with dualizing sheaf $\omega_f$. Suppose that fibers containing closed points of $Y$ are geometrically $F$-injective and that fibers over the generic points of $X$ (\ie generic fibers) are $F$-rational. Then
\[
\uptau(\omega_Y)= \omega_f \otimes f^\ast \uptau(\omega_X),
\]
which entails that $Y$ is reduced and Cohen--Macaulay. Thus, if $X$ is $F$-rational then so is $Y$.
\end{mainthm*}

Our proof has two main steps. First, we make essential use of duality theory (for Cohen–Macaulay maps) to express explicitly the Cartier operator on $Y$ in terms of the Cartier operator of $X$ and the relative Cartier operator. This is why our result needs the finite-type hypothesis on $f$. Then, the $F$-injectivity hypothesis on the fibers is engineered for the relative Cartier operator to be surjective (it is worth noting that, when $X$ is of finite type, the fibers containing closed points are exactly the fibers over closed points of $X$). The formula then follows by direct calculation once one constructs a common test element for $X$ and $Y$. The $F$-rationality hypothesis on generic fibers is used to ensure this, as well as for $Y$ to also be reduced and Cohen–Macaulay.

We note that Enescu's result does not require the morphism to be Cohen–Macaulay (in fact, much of Enescu's paper is devoted to showing that the relative Frobenius is \emph{parameter pure} in the situation he considers); thus, our assumption about the morphism can possibly be weakened. We also note that Enescu has relaxed assumptions when the base is Gorenstein. It is plausible that our hypothesis on fibers can also be weakened in this case.

\begin{question}
    In a slightly different direction, it is also natural to ask how the $F$-rational signature behaves in such a situation. The $F$-rational signature is another prominent invariant that quantifies $F$-rationality; see \cite{HochsterYaoRationalSignature}, \cite{SannaiDualFsignature}, and \cite{SmirnovTuckerFrationalSignature} for background on this notion. This is probably difficult. For instance, the invariance of the $F$-rational signature for a regular morphism was only proven recently by Shiji Lyu \cite[Theorem 7.2]{LyuGammaConstructionAndRationalSignature}.
\end{question}

For the rest of this note, all schemes and rings are defined over $\mathbb{F}_p$ (i.e., they have equal characteristic $p>0$). Given a scheme or ring $X$, we denote the $e$-th iterate of its (absolute) Frobenius endomorphism by $F^e=F^e_X\:X \to X$, and we use the usual shorthand notation $q \coloneqq p^e$. It is said that $X$ is $F$-finite if $F$ is a finite morphism.

\subsubsection*{Acknowledgments} The authors are very grateful to the anonymous referee for their feedback and valuable suggestions. In particular, the use of the recent work by Bhatt--Blickle--Schwede--Tucker \cite{BlickleBhattSchwedeTuckerDualizingFfinite} to relax the hypotheses in the Main Theorem from $X/\kay$ being of finite-type to $X$ being $F$-finite.

\section{Preliminaries}
In this section, we briefly explain the ingredients and concepts involved in the formulation of our Main Theorem. We start with the following basic observation. 

\begin{lemma}
\label{AssPrimesCanonicalMinimal}
Let $R$ be an $F$-finite reduced Cohen–Macaulay ring with canonical module $\omega_R$. For $\p \in \Spec R$ with residue field $\kappa(\p)$, the following statements hold.
\begin{enumerate}
\item If $\p$ is non-minimal, then $\Hom_{R_\p}(\kappa(\p), \omega_{R_\p}) = 0$. 
\item If $\p$ is minimal, then $\Hom_{R_\p}(\kappa(\p), \omega_{R_\p}) \cong R_\p$.
\end{enumerate}
\end{lemma}
\begin{proof}
Suppose that $\p$ is minimal. Then, $\kappa(\p)=R_{\p}$ as $R$ is reduced. Hence, $\omega_{R_\p} = R_\p =\kappa(\p)$ and so part (b) follows at once. 

Assume now that $\p$ is not minimal. Since $R_\p$ is Cohen–Macaulay, it contains a regular element $x$ that is not a unit. Let $\varphi\: \kappa(\p) \to \omega_{R_\p}$ be a  non-zero $R_{\p}$-linear homomorphism. Composing it with the canonical quotient $R_\p \to \kappa(\p)$ yields a non-zero map $\varphi'\colon R_\p \to \omega_{R_\p}$ with $\ker \varphi' \supset \p$. However, $x \notin \ker \varphi'$ as $\omega_{R_{\p}}$ is $\textbf{S}_1$ (see \cite[Lemma 1.2]{HartshonreGeneralizedDivisorsAndBiliaison}), which is a contradiction.
\end{proof}

\autoref{AssPrimesCanonicalMinimal} shows that the definition of the adjoint test module $\uptau(\omega_X)$ in the sense of \cite[Definition 1.10]{BlickleStablerFunctorialTestModules} coincides with the one given in \cite[Definition 3.1]{BlickleTestIdealsViaAlgebras}, which we recall next.

\begin{definition}[Adjoint test modules and related notions]
\label{def.canonicalsheafandtestmodule}
Let $X$ be a reduced $F$-finite Cohen–Macaulay scheme.
\begin{enumerate}[(a)]
    \item Throughout we work with a fixed \emph{canonical/dualizing sheaf} $\omega_X$ admitting an  isomorphism
    \[ \xi \coloneqq \xi_e \coloneqq \omega_X \xrightarrow{\sim} F^{e,!} \omega_X\] which exists by \cite[Theorem 1.0.1]{BlickleBhattSchwedeTuckerDualizingFfinite}
    Composing with the Frobenius trace
    \[
\Tr \coloneqq \Tr^e \coloneqq  \Tr_{F^e_X, \omega_X} \:F^e_{*,X} F^{e,!}_X \omega_X \to \omega_X
    \]
    yields the ($e$-th power of the) \emph{Cartier operator} on $X$
    \[
    \kappa^e \coloneqq \kappa^e_X \coloneqq \Tr \circ F^e_{X,*} \xi \: F^e_{X,*} \omega_X \to \omega_X
    \]
    In the above, we used $\xi$ to construct $\kappa=\kappa^1$, but in reality, one determines the other. In fact, $\xi$ can be recovered from $\kappa$ through the formula
    \[
    \xi(m) = [r \mapsto \kappa^e(F^e_*rm)]
    \]
on local sections. In other words, $\kappa$ and $\xi$ are adjoint to one another with respect to $F^e_* \dashv F^{e,!}$. This defines a Cartier module $(\omega_X,\kappa)$ on $X$.
    \item $X$ is said to be \emph{$F$-injective} if $\kappa$ is surjective.
    \item The \emph{adjoint test module} $\uptau(\omega_X)$ is defined as the smallest Cartier submodule $\sF \subset \omega_X$ (i.e., the smallest $\sO_X$-submodule for which $\kappa(F_*\sF) \subset \sF$) such that $\sF_\eta = \omega_{X, \eta}$ at every generic point $\eta \in X$.
    \item $X$ is said to be \emph{$F$-rational} if $\uptau(\omega_X) = \omega_X$.
    \item In particular, the existence of $\uptau(\omega_X)$ may be checked locally on an open affine covering, and we may glue to obtain $\uptau(\omega_X)$. Letting $\Spec R \subseteq X$ be such an affine chart, there is $c \in R$ avoiding every minimal prime such that $\uptau(\omega_{R_c}) = \omega_{R_c}$ and for any such $c$ we have 
    \[ \uptau(\omega_R) = \sum_{e \geq e_0} \kappa^e(F_*^e c^t \omega_R)\]
    for $t \gg 0$ and $e_0 \gg 0$; see \cite[Theorem 3.11, Corollary 3.16]{BlickleTestIdealsViaAlgebras}. Such an element $c$ is called a \emph{test element}.
\end{enumerate}
\end{definition}

\begin{definition}
A finite-type morphism of locally noetherian schemes is \emph{Cohen–Macaulay} if it is flat and its fibers are Cohen–Macaulay (in particular, locally noetherian).
\end{definition}

\begin{definition}
A finite-type morphism is said to be \emph{of relative dimension $n$} if every non-empty fiber has dimension $n$.
\end{definition}

\begin{setup}\label{set.frobeniustrace}
    With notation as in \autoref{def.canonicalsheafandtestmodule}, if $f\: Y \to X$ is a Cohen–Macaulay morphism of relative dimension $n$, then it comes equipped with a dualizing sheaf $\omega_f \coloneqq f^! \sO_X$; see \cite[p.157 (after Corollary 3.5.2)]{ConradGDualityAndBaseChange} for more details. Moreover, we can define 
    \[
    \omega_Y \coloneqq f^! \omega_X = \omega_f \otimes f^\ast \omega_X,
    \]
    where we use the fact that $f$ is Cohen--Macaulay of relative dimension $n$ and that $X, Y$ are Cohen--Macaulay for the last equality; see \cite[Theorem 4.3.3]{ConradGDualityAndBaseChange}\footnote{Note that a residual complex corresponds to a dualizing complex---see \cite[Chapter VI, Introduction]{HartshorneResidues}.}, \cite[\href{https://stacks.math.columbia.edu/tag/0C10}{Tag 0C10}]{stacks-project}. 
    
    Since $f \circ F^e_Y = F^e_X \circ f$ we obtain, using \cite[Theorem 1.0.1]{BlickleBhattSchwedeTuckerDualizingFfinite} an isomorphism
    \[ 
    \zeta\colon \omega_Y = f^! \omega_X \xrightarrow{\sim} f^! F^{e,!}_X  \omega_X \xrightarrow{\sim} (F_X^e \circ f)^! \omega_X \xrightarrow{\sim} (f \circ F^e_Y)^! \omega_X \xrightarrow{\sim} F^{e,!}_Y f^! \omega_X  = F_Y^{e,!} \omega_Y.
    \]
    Composing it with the appropriate Frobenius trace yields the Cartier operator $\lambda^e =\lambda_Y^e \: F^e_* \omega_Y \to \omega_Y$. In other words, we define $\lambda$ to be adjoint to $\zeta$. Of course, under this isomorphism, $\lambda^e_Y$ would correspond to the Cartier operator $\kappa_Y^e$ defined as in \autoref{def.canonicalsheafandtestmodule}. However, for the transformation rule to be an equality on the nose rather than an isomorphism, we need to define $\omega_Y$ and its Cartier operator in this way.
\end{setup}

\section{Proof of the Main Theorem}

To prove our Main Theorem, we need to first establish two lemmas. The first one can be extracted from \cite[Theorem 4.3.3]{ConradGDualityAndBaseChange}, and we need it to show our second key lemma. We include a proof for the convenience of the reader.

\begin{lemma}
\label{dualitylemma}
Let $f\colon \Spec S \to \Spec R$ be a Cohen–Macaulay morphism of relative dimension $n$ that factors as $f =g \circ a$, where $g\colon \Spec T \to \Spec R$ is Cohen–Macaulay and $a \colon \Spec S \to \Spec T$ is finite dominant.  We have a natural isomorphism $a^! \omega_g \otimes_S a^\ast g^\ast  \to a^! \circ (\omega_g \otimes_T g^\ast )$ given by the canonical isomorphisms
\begin{equation*}
    \begin{split}
    \Hom_T(a_{\ast}S, \omega_{g}) \otimes_{S} a^\ast g^\ast M &\longrightarrow \Hom_T(a_{\ast}S, \omega_{g} \otimes_T g^\ast M)\\
    [s \mapsto \varphi(s)] \otimes s' \otimes t \otimes m &\longmapsto [s \mapsto \varphi(ss') \otimes t \otimes m].
    \end{split}
\end{equation*}
\end{lemma}
\begin{proof}
Since $a$ is finite dominant, we may factor it as $\Spec S \to \Spec T/I \to \Spec T$, where $I$ is nilpotent and $T/I \subset S$ is finite injective. In particular, $a$ is of relative dimension zero. We conclude that $g$ is of relative dimension $n$.

Since $g$ is of finite type, we may factor it as 
\[
\Spec T \xrightarrow{i} \Spec P \xrightarrow{\pi} \Spec R,
\]
where $i$ is a closed immersion and $\pi$ is smooth. By duality theory, we have 
\begin{equation}\label{eq.dualizing}\omega_{g} = \Ext^{n}_P(i_\ast T, \omega_{\pi})
\end{equation}
and, since $g$ is Cohen–Macaulay, $n$ is the only degree where the $\Ext$ is non-vanishing. By (derived) tensor-hom adjunction, we have an isomorphism
\[
\RHom_T(a_{\ast}S, \RHom_P(i_\ast T, \omega_{\pi})) \cong \RHom_P(i_\ast a_{\ast} S, \omega_{\pi}).
\]
Therefore the right hand term is computed by a Grothendieck spectral sequence $E_2^{p,q}$. By the above, the only non-zero $E_2$-terms are of the form $E_2^{p,n}$, whence the sequence collapses at the $E_2$-page.
Moreover, since $f$ is Cohen–Macaulay, the right-hand side is supported in a single degree $n$. In summary, we obtain the following isomorphism from tensor-hom adjunction:
\[   
    \Hom_T(a_{\ast}S, \Ext_P^n(i_\ast T, \omega_{\pi})) \cong \Ext^n_P(i_\ast a_{\ast} S, \omega_{\pi}).
\]

Note that, due to the previous spectral sequence analysis, we have that 
\[
\Ext^1_T(a_{\ast}S, \Ext_P^n(i_\ast T, \omega_{\pi})) = 0
\] 
In particular, $\omega_g = \Ext_P^n(i_\ast T, \omega_{\pi})$ and $S$ are both $R$-flat. Moreover, if we pass to a fiber $g_x\colon \Spec T_x \to \Spec \kappa(x)$ of $g$, then a similar spectral sequence argument shows that 
\[
\Ext^1_{T_x}(a_{x,\ast} S_x, \omega_{g_x}) = 0,
\] 
where $\omega_{g_x}$ coincides with the pullback of $\omega_g$ to $T_x$. Applying \cite[Theorem 1.9, (ii) (3)$\Rightarrow$(2)]{altmankleimancompactifyinpicardscheme} with $g = \id$, we conclude that the canonical map
\begin{equation}\Hom_T(a_{\ast}S, \Ext_P^n(i_\ast T, \omega_{\pi})) \otimes_S a^\ast g^\ast M \longrightarrow \Hom_T(a_{\ast}S, \Ext_P^n(i_\ast T, \omega_{\pi}) \otimes_T  g^\ast M)
\end{equation}
is a natural isomorphism. With \autoref{eq.dualizing}, the proof is complete.
\end{proof}

The next lemma will express the Cartier operator of $Y$ in terms of those of $X$ and $f$. To explain what the latter is, let us recall the cartesian diagram defining the relative Frobenius of a map $f\: Y \to X$:
\begin{equation}\label{relFrobenius}
\xymatrixcolsep{5pc}\xymatrixrowsep{3pc}\xymatrix{
Y \ar@/_/[ddr]_-{f} \ar@/^/[drr]^-{F^e_Y} \ar@{.>}[dr]|-{F^e_{f}}\\
&Y^{(q)} \ar[d]_-{f^{(q)}} \ar[r]^-{G^e_f}  & Y\ar[d]^-{f} \\
&X \ar[r]^-{F^e_X} &X}
\end{equation}
By base change (\cite[Theorem 3.6.1]{ConradGDualityAndBaseChange}), we identify $G_f^{e,\ast} \omega_f$ with $\omega_{f^{(q)}}$. Under such identification, the composition $f = f^{(q)} \circ F_f^e$ induces an isomorphism 
\[\Gamma\: \omega_f \xrightarrow{\sim} F_f^{e,!} G_f^{e,\ast} \omega_{f}.
\]
We write $\gamma\coloneqq \gamma^e\: F_{f,\ast}^e \omega_f \to G^{e\ast}_f \omega_f$ for its adjoint and refer to this map as the \emph{($e$-th) relative Cartier operator}. We then have 
\[
\Gamma(\delta) = [s \mapsto \gamma(s\delta)].
\]

\begin{lemma}
\label{le.CMtracefacorization}
With notation as in \autoref{set.frobeniustrace} and \autoref{relFrobenius}, the Cartier operator $\lambda\:F_{Y,\ast}^e \omega_Y \to \omega_Y$ admits the following factorization.
\begin{align}\label{eq.tracefactorizationCM}
    F_\ast^e \omega_Y= G^e_{f,\ast} F^e_{f,\ast} \omega_Y = G^e_{f,\ast} F^e_{f,\ast} (\omega_f \otimes f^\ast \omega_X) = {}& G^e_{f,\ast} F^e_{f,\ast} (\omega_f \otimes F_f^{e,\ast} f^{(q),\ast} \omega_X)\\ \nonumber
     \xleftarrow[\sim]{\text{(affine) projection formula}} {} & G^e_{f,\ast}( F^e_{f,\ast}(\omega_f) \otimes f^{(q),\ast} \omega_X) \\\nonumber
     \xrightarrow{G^e_{f,\ast}(\gamma \otimes \id)} {} & G^e_{f,\ast} (\omega_{f^{(q)}} \otimes f^{(q),\ast} \omega_X) \\\nonumber
     = {}&  G^e_{f,\ast} (G^{e,\ast}_f \omega_f \otimes f^{(q),\ast} \omega_X) \\ \nonumber
     \xleftarrow[\sim]{\text{(affine) projection formula}} {} & \omega_f \otimes G^e_{f,\ast} f^{(q),\ast} \omega_X\\\nonumber
     \xleftarrow[\sim]{canonical} {} &
\omega_f \otimes f^\ast F^e_\ast \omega_X \\\nonumber
\xrightarrow{\id \otimes f^\ast \kappa}{}&\omega_f \otimes f^\ast \omega_X = \omega_Y.
\end{align}
\end{lemma}
\begin{proof}
We may verify this locally and thus assume that $f$ is an affine map $\Spec S \to  \Spec R$. We let $T$ be such that $\Spec T = Y^{(q)}$. By duality theory, we have an isomorphism 
\[
f^! F_X^{e,!} \omega_X \xrightarrow{\sim} F^{e,!}_Y f^! \omega_X = F^{e,!}_Y \omega_Y.
\] 
We assert that the adjoint of the composition \begin{equation}\label{eq.dualitycomposition}
    \omega_Y \xrightarrow{f^! \xi} f^! F^{e,!}_X \omega_X \xrightarrow{\sim} F_Y^{e,!} \omega_Y
\end{equation}
is the sequence of compositions \autoref{eq.tracefactorizationCM}. First, we describe the composition \autoref{eq.dualitycomposition} explicitly. 

\begin{claim} The composition in \autoref{eq.dualitycomposition} is given by the isomorphism
\begin{equation*}
\begin{split}
\omega_f \otimes f^\ast \omega_X &\longrightarrow F_X^{e,!}\big(\omega_f \otimes f^\ast \omega_X\big)\\
\delta \otimes s \otimes m &\longmapsto \biggl[t \mapsto \sum_i \delta_i \otimes v_i \otimes \kappa(r_i m)\biggr]
\end{split}
\end{equation*}
where we set $\gamma(st \delta) = \sum_i \delta_i \otimes r_i \otimes v_i \in \omega_f \otimes_S T$.    

\begin{proof}[Proof of the claim]

We start with the isomorphism
\begin{equation*}
\begin{split}
    \omega_Y = \omega_f \otimes f^\ast \omega_X &\xrightarrow{\id \otimes f^\ast \xi} f^! F_X^{e,!} \omega_X\\
    \delta \otimes  s \otimes m &\xmapsto{\phantom{\id \otimes f^\ast K}} \delta \otimes s \otimes [r \mapsto \kappa(rm)].
\end{split}
\end{equation*}
Now, we compose with $\Gamma \otimes \id$ to obtain an isomorphism
\begin{equation*}
    \begin{split}
        f^! F_X^{e,!} \omega_X = \omega_f \otimes_S f^\ast F_X^{e,!} \omega_X &\longrightarrow F^{e,!}_f G_f^{e,\ast} \omega_f \otimes_S f^\ast F_X^{e,!} \omega_X\\
        \delta \otimes s \otimes [r \mapsto \kappa(rm)] &\longmapsto [t \mapsto \gamma(t\delta)] \otimes s \otimes [r \mapsto \kappa(rm)].
    \end{split}
\end{equation*}
Since $f^\ast = F_f^{e\ast} f^{(q),\ast}$, we obtain an isomorphism
\begin{align*}
    F^{e,!}_f G_f^{e,\ast} \omega_f \otimes_S f^\ast F_Y^{e,!} \omega_X &\longrightarrow F^{e,!}_f G_f^{e,\ast} \omega_f  \otimes_S F_f^{e\ast} f^{(q),\ast} F_X^{e,!} \omega_X =\\& F^{e,!}_f G_f^{e,\ast} \omega_f \otimes_{F^e_\ast S} F^e_\ast S \otimes_T T \otimes_{F_\ast^e R} F^{e,!}_X \omega_X\\
    [t \mapsto \gamma(t\delta)] \otimes s \otimes [r \mapsto \kappa(rm)] &\longmapsto [t \mapsto \gamma(t\delta)] \otimes s \otimes 1 \otimes [r \mapsto \kappa(rm)].
\end{align*}
We now apply \autoref{dualitylemma} to produce the isomorphism
\begin{equation*}
    \begin{split}
    F^{e,!}_f G_f^{e,\ast} \omega_f  \otimes_S F_f^{e\ast} f^{(q),\ast} F^{e,!} \omega_X &\longrightarrow F^{e,!}_f \big(G_f^{e,\ast} \omega_f  \otimes_T f^{(q),\ast} F^{e,!} \omega_X\big)\\
    [t \mapsto \gamma(t\delta)] \otimes s \otimes 1 \otimes [r \mapsto \kappa(rm)] &\longmapsto 
    \big[t \mapsto \gamma(st\delta) \otimes 1 \otimes [r \mapsto \kappa(rm)]\big] .  \end{split}
\end{equation*}
Since $f$ is flat and $F^{e,!} \omega_X = \Hom_R(F_\ast^e R, \omega_X)$, we can use flat base change for hom to write the isomorphism
\begin{equation*}
    \begin{split}
        F^{e,!}_f \big(G_f^{e,\ast} \omega_f  \otimes_T f^{(q),\ast} F^{e,!} \omega_X\big) &\longrightarrow F_f^{e,!} \big(G_f^{e,\ast} \omega_f  \otimes_T G_f^{e,!} f^\ast \omega_X\big)\\
         \big[t \mapsto \gamma(st\delta) \otimes 1 \otimes [r \mapsto \kappa(rm)]\big]
        &\longmapsto \Big[t \mapsto \big[\gamma(st\delta) \otimes [r \otimes v \mapsto \kappa(rm) \otimes v]\big] \Big].
    \end{split}
\end{equation*}

Let us now write \[
\gamma(st\delta) = \sum_i \delta_i \otimes r_i \otimes v_i \in G_f^{e\ast} \omega_f = \omega_f \otimes_{S} T.
\]
Then, using $T$-linearity, we may rewrite the image of the above as
\[
t \mapsto \sum_i \delta_i \otimes [r \otimes v \mapsto \kappa(rr_im) \otimes vv_i].
\]
We apply \autoref{dualitylemma} once more to get an isomorphism
\begin{equation*}
    \begin{split}
        F_f^{e,!} \big(G_f^{e\ast} \omega_f  \otimes_T G_f^{e,!} f^\ast \omega_X\big) &\longrightarrow F_f^{e,!} G_f^{e,!} \big(\omega_f \otimes f^\ast \omega_X\big)\\
        \Bigg[t \mapsto \bigg[\sum_i \delta_i \otimes [r \otimes v \mapsto \kappa(rr_im) \otimes vv_i]\bigg] \Bigg]&\longmapsto \Bigg[t \mapsto \bigg[r \otimes v \mapsto \sum_i \delta_i \otimes v v_i \otimes \kappa(rr_i m)\bigg] \Bigg].
    \end{split}
\end{equation*}
The isomorphism $F_f^{e,!} G_f^{e,!} \cong F_Y^{e,!}$ is given by tensor-hom adjunction and yields the isomorphism
\begin{equation*}
    \begin{split}
     F_f^{e,!} G_f^{e,!} \big(\omega_f \otimes f^\ast \omega_X\big) &\longrightarrow F_X^{e,!}\big(\omega_f \otimes f^\ast \omega_X)\big)\\
     \Bigg[t \mapsto \bigg[r \otimes v \mapsto \sum_i \delta_i \otimes v v_i \otimes \kappa(rr_i m)\bigg] \Bigg]&\longmapsto \Bigg[rta^q \mapsto \sum_i \delta_i \otimes v v_i \otimes \kappa(r r_i m)\Bigg].
    \end{split}
\end{equation*}
Composing all these isomorphisms gives the desired formula.
\end{proof}
\end{claim}
From this, we see that the adjoint morphism of \autoref{eq.dualitycomposition} is given by
\begin{equation}
    \begin{split}
      \label{eq.tracefactorizationexplicitCM}
      F_{Y\ast}^e \big(\omega_f \otimes f^\ast \omega_X\big) &\longrightarrow  \omega_f \otimes f^\ast \omega_X\\
      F_{Y\ast}^e\big(\delta \otimes s \otimes m\big) &\longmapsto \sum_i \delta_i \otimes v_i \otimes \kappa(r_i m).
    \end{split}
\end{equation}
This is readily seen to be the same as \autoref{eq.tracefactorizationCM}.
\end{proof}

We now have all the ingredients to proceed with our proof.

\begin{proof}[Proof of Main Theorem.]
The statement is local on $X$, so we may work in the affine case $f\: \Spec S \to \Spec R$. Note that $S$ is both Cohen–Macaulay and reduced. Indeed, it is Cohen–Macaulay by \cite[\href{https://stacks.math.columbia.edu/tag/045J}{Lemma 045J}]{stacks-project}. We only need to verify that it is $\mathbf{R}_0$; this follows from \cite[\href{https://stacks.math.columbia.edu/tag/031E}{Lemma 031E}]{stacks-project} and our hypothesis on generic fibers.

\begin{claim}
    There is $c \in R$ that is a common test element for $\omega_R$ and $\omega_S$.
\end{claim}
\begin{proof}[Proof of claim]
    By the $F$-rationality of the generic fibers, for each minimal prime $\p$ of $\Spec R$, there is $x_\p \notin \p$ such that $\uptau(\omega_{S} \otimes R_{x_\p}) = \omega_S \otimes R_{x_\p}$. By prime avoidance, there exists $x \in R$ such that $D(x) \subset \bigcup_\p D(x_\p)$ with $D(x)$ containing all the generic points $\p$ of $\Spec R$. It follows that $\uptau(\omega_S \otimes R_x) = \omega_S \otimes R_x$. By going down for flat maps, minimal primes of $S$ are mapped to minimal primes. Hence, $f^{-1}(D(x))$ is dense in $\Spec S$. It follows that $x$ is a test element for $\omega_S$. If $y$ is a test element for $\omega_R$, then $c = xy$ is the desired common test element.
\end{proof}

\begin{claim}
   The relative Cartier operator $\gamma \: F_{f,\ast}^e \omega_f \to \omega_{f^{(q)}}$ is surjective.
\end{claim}
\begin{proof}[Proof of claim]
    Since $f$ is  Cohen–Macaulay, its Cartier operator is compatible with base change (see \cite[Lemma 2.16]{PatakfalviSchwedeZhangFsingsinFamilies}). Thus, the base change of $\gamma$ to any perfect fiber $\Spec S \otimes \kappa(\p)_{\mathrm{perf}} \to \Spec \kappa(\p)_{\mathrm{perf}}$ is the corresponding relative Cartier operator. Since the base is now perfect, we may identify this with the absolute Cartier operator. If $f^{-1}(\p)$ is a fiber containing a closed point, then such a Cartier operator is surjective by our assumption that such fibers are geometrically $F$-injective. Since $\kappa(\p) \to \kappa(\p)_\mathrm{perf}$ is faithfully flat, it follows that the relative Cartier operators on ordinary closed fibers are also surjective. Nakayama's lemma implies that $\gamma$ is surjective.
\end{proof}

Using the factorization of $\lambda$ obtained in \autoref{le.CMtracefacorization} and the common test element $c \in R$, from the surjectivity of $\gamma$, we get that 
\[
\uptau(\omega_Y) =  \sum_{e \geq e_0} \lambda^e(F^e_* c^t \omega_Y) = \sum_{e \geq e_0} (\id \otimes f^\ast \kappa^e)(\omega_f \otimes c^t  f^\ast F^e_*\omega_X) = \omega_f \otimes f^\ast \uptau(\omega_X),
\] 
and the transformation rule is proven.
\end{proof}

\bibliographystyle{skalpha}
\bibliography{MainBib}

\end{document}